\documentclass{amsart}

\usepackage{amsmath,amssymb,amsthm,amsfonts,amscd,array,graphicx}
\usepackage{mathrsfs}

\usepackage[T2A]{fontenc}     
\usepackage[cp1251]{inputenc} 
\usepackage[english]{babel}    

\newcommand{\K}{\mathbb{K}}
\newcommand{\Z}{\mathbb{Z}}
\newcommand{\Aut}{\mathrm{Aut}}

\newcommand{\Y}{\mathrm{Y}}
\newcommand{\A}{\mathcal{A}}

\newtheorem{theorem}{Theorem}[section]

\newtheorem{cor}[theorem]{Corollary}
\newtheorem{lem}[theorem]{Lemma}
\newtheorem{prop}[theorem]{Proposition}
\theoremstyle{definition}
\newtheorem{defin}[theorem]{Definition}
\newtheorem{remark}[theorem]{Remark}

\author{Anton Trushin}

\title{Gradings allowing wild automorphisms}

\date{25.03.2021}

\address{Lomonosov Moscow State University, Faculty of Mechanics and Mathematics, Moscow, Russia}
\address{Moscow Center for Fundamental and Applied Mathematics Moscow Russia}


\email{ornkano@mail.ru}

\begin{document}
	
	\maketitle
	
	\begin{abstract}
		In 2004 Shestakov and Umirbaev proved that the Nagata automorphism of the polynomial algebra in three variables is wild. We fix a {$\Z$-grading} on this algebra and consider graded-wild automorphisms, i.e. such automorphisms that can not be decomposed onto elementary automorphisms respecting the grading. We describe all gradings allowing graded-wild automorphisms.
	\end{abstract}

	\section{Introduction}
	Let $\K$ be an algebraically closed field of characteristic zero and let $\A=\K[x_1,\ldots,x_n]$ denote the polynomial ring in $n$ variables over $\K$.
	A {\it polynomial automorphism} is a bijective homomorphism $\varphi = (f_1,\ldots,f_n): \A \rightarrow \A$ of the form
	$$\varphi(F(x_1,\ldots,x_n)) = F(f_1(x_1,\ldots,x_n),\ldots,f_n(x_1,\ldots,x_n)),$$
	where each $f_i$ belongs to $\A$. 
	The following definitions plays a crucial role in the sequel.
	\begin{defin}
		An {\it affine} automorphism of the algebra $\A$ is an automorphism $\varphi$ of the form $x_i \mapsto \sum\limits_{j = 1}^nc_{ij}x_j + b_i$. If all $b_i=0$, we call $\varphi$ {\it linear}.  
	\end{defin}
	\begin{defin}
		An automorphism of the algebra $\A$ is {\it elementary} if it has the following form $$\varphi = (x_1, \ldots, x_{i-1},\ x_i + F,\ x_{i+1}, \ldots, x_n),$$	
		where $F \in \K[x_1, \ldots, x_{i-1},x_{i+1}, \ldots, x_n]$.
	\end{defin}
	We call an automorphism $\varphi$ {\it tame} if it is the composition of elementary and linear automorphisms.
	
	\begin{remark}\label{mr}
		Each elementary automorphism can be represented as a composition of an automorphisms of the form $$\varphi = (x_1, \ldots, x_{i-1},\ \lambda x_i + M,\ x_{i+1}, \ldots, x_n),$$
		where $M \in \K[x_1, \ldots, x_{i-1},x_{i+1}, \ldots, x_n]$ is monomial.
	\end{remark}
	An automorphism is said to be {\it wild} if it is not tame.
	The problem of existence of wild automorphisms is now solved for spaces of dimension less than four.\\
	
	$\bullet$ For algebra $\K[x]$ automorphism group consists of affine automorphisms: $x \mapsto \lambda x + c$. They are all tame.\\
	
	$\bullet$ In 1942 Jung proved that the algebra $\K[x,y]$ admits no wild automorphisms \cite{b5}, see also \cite{K}. Moreover, the automorphism group is an amalgamated product of subgroups, one of which is a subgroup of affine automorphisms \cite{b1}.\\
	
	$\bullet$ In 1972 Nagata constructed the following automorphism $\sigma$ of $\mathbb{K}[x,y,z]$
	$$
	\begin{cases}
	\sigma(x)=x+(x^2-yz)z;\\
	\sigma(y)=y+2(x^2-yz)x+(x^2-yz)^2z;\\
	\sigma(z)=z.
	\end{cases}
	$$
	In 2004 Shestakov and Umirbaev  \cite{b2} proved that the Nagata automorphism gives an example of a wild automorphism for $n = 3$.
	\\
	
	But existence of a wild automorphism for $n=3$ does not imply existence of a wild automorphism for $n>3$. If we consider the automorphism of the polynomial algebra in four variables obtained by adding to the Nagata automorphism an extra stable variable, it turns out to be tame, see \cite{S}. So the Nagata automorphism does not give an example of a wild automorphism for $n=4$. The problem of existence of wild automorphisms for $n=4$ is still open. 
	
	In paper \cite{b3} tame and wild automorphisms of the algebra $\A$ with an additional $\Z$-grading is considered. Let us recall some definitions.
	Let $G$ be a commutative group. Then a \mbox{\it $G$-grading} $\varGamma$ of the algebra  $\A$ is a decomposition $\A$ into a direct sum of linear subspaces~$\A = \bigoplus\limits_{g \in G}\A_g$, such that~$\A_{g}\A_{h} \subset \A_{gh}$. Subspaces $\A_g$ are called {\it homogeneous components} and all elements of $\A_g$ are called {\it homogeneous}. If $f$ belongs to $\A_g$ we denote $\deg_{\varGamma}(f) = g$. If $G=\Z$, for an arbitrary $f = \sum\limits_{g}f_g, f_g \in \A_g$ we denote by $\deg_{\varGamma}(f)$ the maximum $g$ such that $f_g \neq 0$. We say that an automorphism $\varphi$ of the algebra $\A$ respects $G$-grading $\varGamma$ if for all $g \in G$ the image of the subspace $\A_g$ under $\varphi$ is contained in $\A_g$. Such an automorphism is called {\it graded}.
	
	In this paper we consider only such gradings $\varGamma$ that all coordinates $x_i$ are homogeneous with respect to $\varGamma$.
	
	\begin{defin}
		If a graded automorphism $\varphi$ can be decomposed to a composition of graded elementary and linear automorphisms then $\varphi$ is called {\it graded-tame}. Other graded automorphisms are called {\it graded-wild}.
	\end{defin}
	Paper \cite{b3} contains an example of a $\Z$-grading for $n=4$ which admits graded-wild automorphisms. In particular the Anick automorphism is graded-wild with respect to this grading.\\
	
	In this paper we study $\Z$-gradings of algebra $\A$ for $n$ equals 2 and 3. We show that if $n=2$, then all graded automorphisms of $\A$ are graded-tame. For $n=3$ we give a criterium for a $\Z$-grading of $\A$ to admit graded-wild automorphisms. For every grading admitting graded-wild automorphisms, we give an explicit example of a graded-wild automorphism.
	
	The author is grateful to Sergey Gaifullin and Dmitry Trushin for useful discussions and comments and to the referee for careful reading of the paper and valuable suggestions.

	\section{Polynomial algebra in two variables}
	For many subsequent reasons, we need both Jung's theorem and its proof, so we give them here.
	\begin{theorem}[Jung]\label{Jung}
		Any automorphism of the algebra $\K[x,y]$ is tame.
	\end{theorem}
	Before proving the theorem, we give some necessary definitions.
	\\
	A {\it derivation} of the algebra $\A$ is a linear mapping $\delta: \A \rightarrow \A$ satisfying the Leibniz law: $\delta(ab) = \delta(a)b + a\delta(b)$.
	A {\it locally nilpotent derivation (LND)} of the algebra $\A$ is a derivation $\delta$ such that for any $a \in \A$ there is a positive integer $m$ such that $\delta^m(a) = 0$.
	
	The {\it Jacobian matrix} for a mapping $\varphi = (f,g) \colon \K[x,y]\rightarrow \K[x,y]$ is the matrix of the partial derivatives of $\varphi$: 
	$
	\begin{pmatrix}
	\frac{\partial f}{\partial x}& \frac{\partial f}{\partial y}\\
	\frac{\partial g}{\partial x}& \frac{\partial g}{\partial y}
	\end{pmatrix}$. 
	The determinant of the Jacobian matrix for $\varphi$ is called the {\it Jacobian} of the mapping $\varphi$ and is denoted by $J(\varphi) = J(f,g)$.

	Now let us consider all monomials having nonzero coefficients in $f$ and mark them as points on the lattice~$\Z^2 $. We assotiate the point $(u,v)$ with the monomial $ \lambda x^{u} y^{v}, \lambda \neq 0 $.
	A {\it Newton polygon} $ N(f) \subset \Z^2 $ is the convex hull of the origin and all the marked points.
	
	Now we give a proof of Jung's theorem due to Makar-Limanov \cite{ML}:
	\begin{proof}[Jung's theorem proof]

		Let $\varphi = (f,g)$ be an automorphism of the algebra $\A=\K[x,y]$. Consider the mapping $\delta_f: \A\rightarrow\A$ such that $\delta_f(h) = J(f,h)$.	
		This map is linear and satisfies the Leibniz rule, so it is a derivation. Note that $ \delta_x (h) = \frac{\partial h}{\partial y} $, that is, $ \delta_x $ is an LND. Since $\varphi$ is an automorphism, we have~$\K[x,y] = \K[f,g]$. The mapping $\delta_{f}$ is a nonzero LND of the algebra~$\K[f,g]$, so $\delta_{f}$ is a nonzero locally nilpotent derivation of the algebra $\A$. \vskip 2pt
		\noindent
		\begin{minipage}{0.65\linewidth}
			Consider the Newton polygon of the polynomial~$f$.
			Let~$\frac {p} {q} $ be a positive irreducible fraction. We consider the following $ \Z $-grading $ \varGamma $ on $ \A $:
			$$\A = \bigoplus\limits_{d \in \mathbb{Z}}{\A}_d,\ \A_d = \langle x ^ uy ^ v|\ qu + pv = d\rangle.$$
			Let us denote by $\widehat{f} $ the top homogeneous component of $ f $ with respect to~$\varGamma$. And by $f_0$ we denote the sum of terms of lower degrees:~$ f_0 = f - \widehat{f} $.
		\end{minipage}
		\hfill
		\begin{minipage}{0.29\linewidth}
			\includegraphics[scale = 0.8]{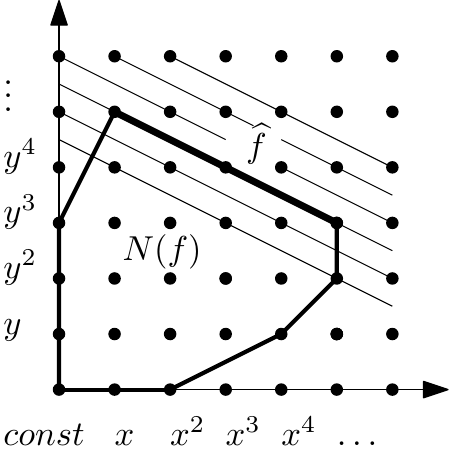}
		\end{minipage} \vskip 3pt
		\noindent
		The derivation $\delta_f$ is locally nilpotent i.e. for any $ h \in \A $ there exists~$m $ such that $\delta_f^m(h) = 0$. 
		But~$J(h_1+h_2,h) = J(h_1,h) + J(h_2,h)$. Therefore, $\delta_{f}(h)  = \delta_{\widehat{f}}(h)  + \delta_{f_0}(h)$.
		Suppose $h$ is homogeneous. Then for all $k$ we have 
		$\delta^k _ {\widehat {f}} (h)$ is the top homogeneous component of $\delta^k_{f}(h)$ with respect to the grading $ \varGamma $. 
		Hence, $ \delta ^ m _ {\widehat{f}} (h) = 0 $. 
		Consequently,~$ \delta _ {\widehat{f}}$ is also an LND.
		Let~$\alpha$ be the maximal number such that $\widehat{f}$ is divisible by $x^\alpha$, and $\beta$ be the maximal number such that $\widehat{f}$ is divisible by $y^\beta$. Then $\widehat{f} = x^\alpha y^\beta \widetilde{f}$. Let $$\widetilde{f} = C_1x^{p_1} + C_2x^{p_2}y^{q_2} + \ldots + C_{k-1}x^{p_{k-1}}y^{q_{k-1}} + C_ky^{q_k}.$$
		Note that the function $\widetilde{f}$ is homogeneous with respect to $\varGamma$.
		Let $p_k = q_1 = 0$. Then for all $i \in \left\lbrace 1,\ldots,k\right\rbrace$ we have
		$$
		\ p_iq + q_ip = p_1q + q_1p = p_1q 
		$$
		Therefore, for all $i \in \left\lbrace 2,\ldots,k\right\rbrace$ we obtain
		$$\frac{p_i - p_1}{q_i} = \frac{-p}{q}.
		$$
		Since $\frac{p}{q}$ is irreducible, for every $i \in \left\lbrace 2,\ldots,k\right\rbrace$ we can fix a positive integer number~$t_i$ such that
		$$(p_i - p_1) = t_i(-p),\ q_i = t_iq.$$
		We consider the rational function~$\frac{\widetilde{f}}{x^{p_1}}$:
		$$\frac{\widetilde{f}}{x^{p_1}} = C_1 + C_2x^{p_2 - p_1}y^{q_2} + \ldots + C_{k-1}x^{p_{k-1} - p_1}y^{q_{k-1}} + C_kx^{-p_1}y^{q_k}=$$
		$$= C_1 + C_2(x^{-p}y^{q})^{t_2} + \ldots + C_{k-1}(x^{-p}y^{q})^{t_{k-1}} + C_k(x^{-p}y^{q})^{t_k}.$$
		Consequently $\frac{\widetilde{f}}{x^{p_1}} = F(x^{-p}y^{q})$, where $F$ is a polynomial in one variable. The polynomial $F$ can be decomposed into linear factors, so $$\frac{\widetilde{f}}{x^{p_1}} = C\prod\limits_{j}(x^{-p}y^{q} - \lambda_j) \Longrightarrow {\widetilde{f}} = C\prod\limits_{j}(y^{q} - \lambda_jx^{p}) \Longrightarrow \widehat{f} = C x^\alpha y^\beta\prod\limits_j(y^{{q}} - \lambda_j x^{{p}}).$$
		
		The derivation  $\delta_{\widehat{f}}$ is an LND, hence $\ker\delta_{\widehat{f}}$ is factorially closed, i.e. $gh\in\delta_{\widehat{f}}$ implies $g\in\delta_{\widehat{f}}$, see \cite[Principle~1(a)]{b4}. 
		But $\delta_{\widehat{f}}(\widehat{f}) = J(\widehat{f},\widehat{f}) = 0$ that is $\widehat{f} \in \ker\delta_{\widehat{f}}$. Thus~$x^\alpha$, $y^\beta \in \ker\delta_{\widehat{f}}$, and for all $i$ we have $(y^q-\lambda_ix^p) \in \ker\delta_{\widehat{f}}$.
		By \cite[Principle 11(e)]{b4} the transcendence degree of $\ker\delta_{\widehat{f}}$ over $\K$ is equal to 1. So the following cases are possible:\\
		
		$\bullet$ $\widehat{f} = C x^\alpha$ or $\widehat{f} = Cy^\beta$;\\
		
		$\bullet$ $\widehat{f} = C(y^{{q}} - \lambda x^{{p}})^k$, $\lambda\neq 0$.\\
		
		Therefore, all the vertices of the polygon lie on the coordinate axes. That is the polygon $N(f)$ is a triangle or a segment. Suppose $N(f)$ is a triangle. Then we can choose $p$ and $q$ such that $\widehat{f} = C(y^{{q}} - \lambda x^{{p}})^k$. We have 
		$$\delta_{\widehat{f}}(h) = J(\widehat{f},h) = J(C(y^{{q}} - \lambda x^{{p}})^k,h) = Ck(y^{{q}}-\lambda x^{{p}})^{k-1}J(y^{{q}} - \lambda x^{{p}},h).$$
		Hence $\overline\delta = \delta_{(y^{{q}} - \lambda x^{{p}})}$ is also a locally nilpotent derivation. We have $\overline\delta(x) = -{{q}}y^{{{q}}-1}$, $\overline\delta(y) = -\lambda px^{{{p}}-1}$. If $p > 1$ and $q > 1$, then $x$ divides $\partial(y)$ and $y$ divides $\partial(x)$. In this case  \cite[Principle 5]{b4} implies $\partial(x) = 0$ or $\partial(y) = 0$. Since the transcendence degree of $\ker\delta_{\widehat{f}}$ over $\K$ is equal to 1 we obtain a contradiction. Therefore, at least one of~$p$ and $q$ is equal to 1.
		If ${{p}} = 1$, then~$\widehat{f} = (y^{{q}} - \lambda x)^k$. Let us consider the automorphism $\psi = (x + \frac{1}{\lambda}y^{{q}}, y)$. Then $\varphi\circ\psi(x) = s$, where~$s = f (x+\frac{1}{\lambda}y^{q},y)$.
		Since this automorphism is homogeneous with respect to the grading $\varGamma$, the leading component of the polynomial $s$ with respect to $\varGamma$ has the form: $(y^{{q}} - \lambda(x + \frac{1}{\lambda}y^{{q}}))^k = (-\lambda x)^k$. Thus $N(s)\subset N(f)$ and one of the vertices is deleted.
		\begin{minipage}{\linewidth}
			\centering
			\includegraphics[scale = 0.8]{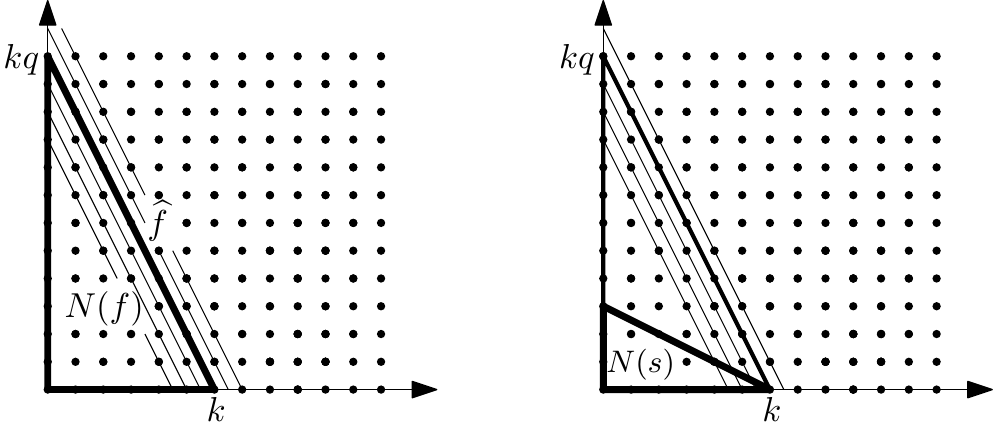}
			\vfill
		\end{minipage}
		So the Newton polygon of the polynomial $s$ is also a triangle (or a segment), but its area is smaller then the area on $N(f)$.
		The case $q = 1$ is similar.
		Now let us prove the claim of the theorem by induction on the area of  $N(f)$. Note that since $N(f)$ is a triangle with integer sides, its area belongs to $\frac{1}{2}\mathbb{Z}$.\\
		
		{\it Base.}\\
		If the area of the polygon $N(f)$ is equal to zero then $f$ has the form $cx + d$ or $cy + d$, where $c\in\mathbb{K}^\times, d\in\mathbb{K}$. If $f = cx + d$, then $J(\varphi) = c\frac{\partial g}{\partial y}$. Since $\varphi$ is an automorphism, $J(\varphi)\in\mathbb{K}^\times$. So $g = \mu y + w(x)$, where $\mu\in\mathbb{K}^\times, w(x)\in\mathbb{K}[x]$. Hence the automorphism~$\varphi$ is tame. Case $f = cy + d$ is similar.\\
		
		{\it Hypothesis.}\\
		Let $n\in\frac{1}{2}\mathbb{Z}$. Suppose that we know that if the area of the polygon $N(f)$ smaller than $n$ then $\varphi$ is tame. \\
		
		{\it Step.}\\
		Now let us show that if the area of $N(f)$ equal to $n$, then $\varphi$ is tame.
		Consider $\psi$ and $s$ introduced above. The area of $N(s)$ smaller than the area on $N(f)$. So by induction hypothesis $\varphi\circ\psi$ is tame. Note that $\varphi = \varphi\circ\psi\circ\psi^{-1}$. But $\psi^{-1}$ is elementary so $\varphi$ can be decomposed into composition of elementary automorphisms, i.e. it is tame.
	\end{proof}
	
	During the proof, we used only automorphisms preserving the origin while the area was nonzero. This implies the following statement.
	
	\begin{cor}\label{ori}
		If the automorphism of the plane preserves $ (0,0) $, then it decomposes into a composition of elementary ones that preserve $ (0,0) $.
	\end{cor}
	
	The following proposition shows that in case of two variables there are no gradings admitting graded-wild automrorphisms. 
	\begin{prop}\label{tgt}
		Let $\widetilde{\varGamma}$ be a grading by an abelian group of the polynomial algebra in two variables. Then all graded automorphisms of $\A=\K[x,y]$ are graded-tame with respect to $\widetilde{\varGamma}$.
	\end{prop}
	\begin{proof}
		Let $\varphi = (f,g)$ be a $\widetilde{\varGamma}$-graded automorphism of $\A$. In the proof of Theorem~\ref{Jung} we consider a grading $\varGamma$ of $\A$ and we show, that $\widehat{f} = C(y^{{q}} - \lambda x)^k$, $\lambda\neq 0$ is the top homogeneous component of $f$ with respect to  $\varGamma$. Also we show that the automorphism
		$\psi = (x + \frac{1}{\lambda}y^{{q}}, y)$ decreases the area of the Newton polygon $N(f)$. Let us prove that $\psi$ is also homogeneous with respcet to $\widetilde{\varGamma}$.
		Since $\varphi$ is homogeneous, $f$ is homogeneous. So all monomials of $\widehat{f}$ has the same degree. But 
		$$\widehat{f} =  C(y^{{q}} - \lambda x)^k = C(y^{qk} - k\lambda xy^{q(k-1)}+\ldots).$$ So $\deg y^{qk} =\deg  xy^{q(k-1)}$. Hence, $\deg x=q \deg y$. Therefore, $\psi$ is homogeneous.
	\end{proof}

	\section{Polynomial algebra in three variables}
	In this section we give a criterium for a $\Z$-grading of the algebra $\A=\K[x,y,z]$ to admit a graded-wild automorphism, see Theorem~\ref{mainmain}.
	
	First of all, we consider the case of strictly positive grading for an arbitrary number of variables.

	\begin{prop}\label{tetriod}
		Let degree of $x_i$ be a positive integer for all $i$.\
		Then all graded automorphisms of the algebra $\K[x_1,\ldots,x_n]$ are graded-tame.
	\end{prop}
	
	\begin{proof}
		Without loss of generality, we can assume that $\deg x_1  \leqslant \ldots \leqslant \deg x_n$. 
		Let $\deg x_1 = \ldots = \deg x_{n_1} = m_1<\deg x_{n_1+1}$ and $\varphi$ be an automorphism preserving this grading.
		Then $x_i$ maps to $ \lambda_1x_1 + \ldots + \lambda_ {n_1} x_ {n_1}$ where $1 \leqslant i \leqslant n_1 $, since the degree of any other monomial is greater than $m_1$.
		
		Now let $\deg(x_{n_1+1}) = \ldots = \deg(x_{n_2}) = m_2<\deg x_{n_2+1}$. 
		Then $x_i$ maps to $L_i(x_{n_1+1}, \ldots, x_{n_2}) + f(x_1, \ldots, x_{n_1})$, where $n_1 < i \leqslant n_2$ , and $L_i$ is a linear function, and $f$ is a polynomial.
		
		Similarly, for variables of degree $m_k$. 
		Let $$\deg x_{n_{k-1}}<\deg x_{n_{k-1}+1} = \ldots = \deg x_{n_k} = m_k<\deg x_{n_k+1}.$$ 
		Then
		$$x_i \mapsto L_i(x_{n_{k-1}+1}, \ldots, x_{n_k}) + f_i(x_1, \ldots, x_{n_{k-1}}), \text{where } n_{k-1} < i \leqslant n_k,$$ 
		Now we represent $\varphi$ as follows: $\varphi =
		\eta_{m_s}\circ\zeta_{m_s}\circ\ldots\circ\eta_{m_2}\circ\zeta_{m_2}\circ\eta_{m_1}$,
		where $\eta_{m_i}$ is a linear mapping changing generators of degree $m_i$, and $\zeta_{m_i}$ is a composition of elementary automorphisms which add polynomials to generators of lower degree to variables of degree $m_i$.
	\end{proof}
	
	\begin{remark}\label{tto}
		If degrees of all variables are negative, then all graded automorphisms are also graded-tame.
	\end{remark}
	
	From now we investigate three-dimensional algebras.
	
	\begin{remark}\label{pm}
		Let us consider a $\Z$-grading $\varGamma$ of $\A=\mathbb{K}[x,y,z]$. If all degrees of variables are positive or negative, then all graded automorphism are graded-tame by Proposition~\ref{tetriod} and Remark~\ref{tto}. Assume this is not the case. If we divide all $\deg_\Gamma x$, $\deg_\Gamma y$ and $\deg_\Gamma z$ by their greatest common divisor, then we obtain a new $\Z$-grading which admits a graded-wild automorphism if and only if $\varGamma$ does. Also multiplying all degrees by $-1$ does not change the situation. Therefore, without loss of generality we can assume $(\deg_{\varGamma}(x), \deg_{\varGamma}(y), \deg_{\varGamma}(z)) = (a,b,-c)$, where $a, b, c\geq 0$, $a \geq b$ and $\gcd(a,b,c)= 1$ or $(a,b,-c)=(0,0,0)$.
	\end{remark}
	
	\begin{theorem}\label{mainmain}
		The grading $\varGamma$ admits graded-wild automorphisms if and only if one of the following occures\\
		
		$\bullet$ $(a,b,c) = (0,0,0)$.\\
		
		$\bullet$ $a,b,c$ strictly positive, $a = qb + pc$ for some integers $p\geq 1$ and $q\geq 2$.\\
	\end{theorem}
	In order to prove this theorem we need some preparations.

	\begin{lem}\label{usl}
		Let $\varphi = (f,\lambda_1y,\lambda_2z)$ is an automorphism of $\A$. Then $f = \lambda x + f_0(y,z)$, where $\lambda\neq 0$.
	\end{lem}
	\begin{proof}
		Let $f(x,y,z) = xf_1(x,y,z) + f_0(y,z)$. Consider such a polynomial $h(y,z)$ that $h(\lambda_1 y, \lambda_2 z)=f_0(y,z)$. Then~$\psi = (x - h(y,z),y,z)$~is also an automorphism. Then $\psi\circ\varphi(x) = xf_1(x,y,z)$. But an automorphism takes an irreducible function to an irreducible one. Hence $f_1 = \lambda\neq 0$.
	\end{proof}
	
	Assume that there is no zeros among numbers $a$, $b$ and $c$. That is $a,b,c>0$. By $\Aut_{\varGamma}(\A)$ we denote the subgroup of $\varGamma$-graded automorphisms in the group $\Aut(\A)$ of all automorphisms.
	
	\begin{lem}\label{otr}
		Let $\varphi$ be a graded automorphism. Then $\varphi(z)=\lambda z$, $\lambda\neq 0$.
	\end{lem}
	\begin{proof}
		Since $\deg_\Gamma(z)<0$ and $\varphi$ preserves the grading, every monomial in $\varphi(z)$ has negative degree. Therefore, $z$ divides $\varphi(z)$. Hence, $\varphi(z)=\lambda z$, $\lambda\neq 0$.
	\end{proof}
	
	\begin{lem}\label{tr}
		If $\gcd( a , c) $ does not divide $ b $, then every graded automorphism of $\A$ has the following form $(\lambda_1x + f_0(y,z),\ \lambda_2y,\ \lambda_3z)$, where $f_0$ is a homogeneous function of degree $a$ and $\lambda_i\neq 0$.
	\end{lem}
	\begin{proof}
		Denote $ \gcd(a,c) = d $. We have $\deg_{\varGamma}(x^pz^q) = ap + cq$. Hence, the degree of a monomial~$x^pz^q$ is a multiple of $d$. Since $d$ does not divide $b$, the image $g$ of $y$ under a $\varGamma$-graded automorphism $\varphi = (f, g, h)$ does not contain monomials of the form $x^pz^q $.
		So every monomial in $g$ is divisible by $y$. Therefore $ g $ is divisible by $y$. But an automorphism takes an irreducible function to an irreducible one. Hence $g = \lambda_2y$ for some nonzero $\lambda_2$. By Lemma \ref{otr} we obtain $h=\lambda_3 z$, $\lambda_3\neq 0$. Thus, the condition of Lemma~\ref{usl} holds, and $\varphi$ preserves the grading of $\varGamma$, so $f=\lambda x+f_0(y,z)$, where $f_0$ is a homogeneous function of degree $a$ with respect to~$\varGamma$. 
	\end{proof}
	
	If conditions of Lemma~\ref{tr} hold then every $\Gamma$-graded automorphism is graded-tame.
	Further, we assume that $\gcd(a, c) = \gcd(b, c) = 1$.


	Let us consider two subgroups of $\Aut_\varGamma(\A)$:
	\begin{equation*}
	\mathrm{E} =  \left\lbrace \varphi \in \Aut_{\varGamma}(\A) \big| \varphi(z) = z \right\rbrace,
	\qquad
	\mathrm{T}=\left\lbrace(x, y, \lambda z)\mid \lambda\neq 0\right\rbrace.
	\end{equation*}
	\begin{lem}The group $\Aut_{\varGamma}(\A)$ is isomorphic to the semidirect product~$\mathrm{E} \leftthreetimes \mathrm{T}$.
	\end{lem}
	\begin{proof}
		By Lemma~\ref{otr}, any graded automorphism has the form $\varphi = (f, g, \lambda z)$. It can be represented as a composition of automorphisms from $\mathrm{E}$ and~$\mathrm{T}$ via $\varphi = \tau\circ\phi$, where~$\tau = (x, y, \lambda z), \phi = (f, g, z)$. Wherein
		~$\tau\circ\phi\circ\tau^{-1} = (x, y, \lambda z)\circ(f, g, z)\circ(x, y, \lambda^{-1}z) = (\widetilde{f},\ \widetilde{g},\ \lambda^{-1}\lambda z) = (\widetilde{f},\ \widetilde{g},\ z) \in \mathrm{E}$.
	\end{proof}
	The subgroup $\mathrm{T}$ consists of linear graded automorphisms. This gives the following assertion.
	\begin{remark}
		Let $\varphi = \tau\circ\phi$, where~$\tau \in \mathrm{T}$.  Then $\varphi$ is graded-tame if and only if $\phi$ is graded-tame.
	\end{remark}
	\begin{remark}
		The plane $\Y = \mathbb{V}(z-1)$ is invariant under automorphisms from the subgroup $\mathrm{E}$. The algebra of regular functions on $Y$ is equal to~$\K[u,v]$, where~$u = x\big|_\Y;\ v = y\big|_\Y$. Therefore, there exists a homomorphism~$\alpha: \mathrm{E} \rightarrow  \Aut(\K[u,v])$. Algebraicaly $\alpha$ is given by the formula
		$$
		\alpha(f,g,z)=(f(u,v,1),g(u,v,1)).
		$$
		We say that the automorphism $\widetilde{\varphi}$ of $\K[u,v]$ can be lifted to an automorphism of $\A$ if the preimage $\alpha^{-1}(\widetilde{\varphi})$ is not empty.
	\end{remark}
	Consider the grading ${\widetilde{\varGamma}}$ of $\K[u,v]$ by the cyclic group $\Z_c$ of order $c$ such that $(\deg_{\widetilde{\varGamma}}(u), \deg_{\widetilde{\varGamma}}(v)) = (\overline{a}, \overline{b})$, where $\overline{a}$ and $\overline{b}$~are the images of $a$ and $b$ under the natural homomorphism from $\Z$ to $\Z_c$.
	
	\begin{lem}
		Let $\varphi \in \mathrm{E}$.
		Then the automorphism $\widetilde{\varphi} = \alpha(\varphi) \in \Aut{(\K[u, v])}$ preserves the grading ${\widetilde{\varGamma}}$.
	\end{lem}
	\begin{proof}
		Let $\varphi = (f, g, z)$. Suppose a monomial $x^iy^jz^k$ has a nonzero coefficient in $f$. Since $\varphi$ preserves the grading, $ai + bj - ck = a$. Then $\widetilde{\varphi} = (\widetilde{f},\widetilde{g})$, where $\widetilde{f} = f(u, v, 1)$. The monomial $x^iy^jz^k$ induces the monomial $u^iv^j$ of $\widetilde{f}$. Therefore, for every monomial $u^iv^j$ of $\widetilde{f}$ we have $i\overline{a}+j\overline{b}=\overline{a}$, that is $\deg_{\widetilde{\Gamma}}\widetilde{f}=\deg_{\widetilde{\Gamma}}u$. Similarly $\deg_{\widetilde{\Gamma}}\widetilde{g}=\deg_{\widetilde{\Gamma}}v$.
	\end{proof}
	\begin{lem}
		Suppose two polynomials $f_1$ and $f_2$ are homogeneous of degree $l$ with respect to $ \varGamma $. If restrictions of $f_1$ and $f_2$  to the plane $\Y = \mathbb{V}(z-1)$ coincide, then~$f_1$ and $f_2$ coincide on the whole space.
	\end{lem}
	\begin{proof}
		Since $f_1\big|_{\Y} = f_2\big|_{\Y}$, then $f_1-f_2$ is divisible by $z-1$. Polynomials $f_1$ and $f_2$ have the same degrees, hence, $f_1-f_2$ is a $\varGamma$-homogeneous element. But $z-1$ is not $\varGamma$-homogeneous. Therefore, $f_1-f_2=0$.
	\end{proof}
	
	\begin{cor}
		Homomorphism $\alpha$ is injective.
	\end{cor}
	It is easy to see that all automorphisms from the group $\mathrm{E}$ descend to automorphisms of the plane.
	However, not all automorphisms of the plane can be lifted to automorphisms of the three-dimensional space.
	
	\begin{lem}\label{ma}
		Let $\widetilde{\varphi} = (\widetilde{f},\widetilde{g}) \in \Aut_{\widetilde{\varGamma}}(\K[u,v])$. Then $\alpha^{-1}(\widetilde{\varphi}) = \varnothing$ if and only if the polynomial $\widetilde{f}$ contains a monomial $v^q$ such that $bq < a$ or the polynomial $\widetilde{g}$ has nonzero free term.
	\end{lem}
	\begin{proof}
		Let assume that $\alpha^{-1}(\widetilde{\varphi}) = \varphi = (f,g,z)$. For each monomial $u^pv^q$ with a nonzero coefficient $\lambda$ in $f\big|_{\Y}$ there is a number $t = \frac{ap+bq-a}{c}$ such that a monomial $x^py^qz^t$ is $\varGamma$-homogeneous of degree $a$. Since $ap + bq \equiv a(\mathrm{mod}\ c)$ number $\frac{ap+bq-a}{c} = t$ is integer. This monomial has the same coefficient $\lambda$ in $f$. Moreover, if a monomial $u^pv^q$ has zero coefficient in $f\big|_{\Y}$, then the monomial $x^py^qz^t$ has zero coefficient in $f$ for each $t$. Finally $t = \frac{ap+bq-a}{c}$ is negative if and only if $p = 0$ and $bq < a$.
		
		Similarly for each monomial $u^pv^q$ in $\widetilde{g}$ there is a number $t = \frac{ap+bq-b}{c}$ and since $a \geq b$ this number negative if and only if $p=q=0$.
	\end{proof}
	
	Since $\gcd(b,c) = 1$ there exists $q$ such that $bq \equiv a (\mathrm{mod}\ c)$. Let us put $\widehat{q} = \max\left\lbrace q \in \Z\big| bq \equiv a (\mathrm{mod}\ c), bq < a\right\rbrace $.
	
	\begin{cor}\label{whq}
		Fix a positive integer $q$ such that $bq \equiv a (\mathrm{mod}\ c)$. Consider a $\widetilde{\varGamma}$-graded automorphism $\widetilde{\varphi} = (u + \lambda v^q,v)$.
		The following conditions are equivalent:\\
		
		$(1)$ $\alpha^{-1}(\widetilde{\varphi}) = (x + \lambda y^qz^t,y,z)$ for some $t \geq 0$.\\
		
		$(2)$ $q > \widehat{q}$.
	\end{cor}

	
	\begin{prop}\label{sz}
		If $\widehat{q} \leqslant 0$, then all graded automorphisms of the algebra $\K[u,v]$ can be lifted to graded automorphisms of $\A$.
	\end{prop}
	\begin{proof}
		All graded automorphisms of the algebra $\K[x,y,z]$ are graded by Proposition \ref{tgt}, and by Corollary \ref{whq} all of them can be lifted to automorphisms of the space.
	\end{proof}
	
	Given a polynomial $F$. Let us denote by~$L(F)$ the linear polynomial obtained by removing all nonlinear terms in $F$.
	
	In order to prove Proposition \ref{re} we need the following lemma.\\
	
	\begin{lem}\label{tra}
		Let $\xi_1 \in \Aut_{\widetilde{\varGamma}}(\K[u,v])$ is an elementary automorphism, $\xi_2 \in \Aut_{\widetilde{\varGamma}}(\K[u,v])$ is a linear automorphism. Then $\xi_2\circ\xi_1 = \zeta_2\circ\zeta_1\circ\zeta_0$, where $\zeta_0$ and $\zeta_2$ are linear automorphisms, $\zeta_1$ is an elementary automorphism, and $\zeta_2(u) = \lambda u$.
	\end{lem}
	\begin{proof}
		Let $\xi_2$ be given by 
		$$\begin{pmatrix}
		u\\
		v
		\end{pmatrix}
		\mapsto
		\begin{pmatrix}
		A& B\\
		C& D
		\end{pmatrix}
		\begin{pmatrix}
		u\\
		v
		\end{pmatrix}
		.$$ If $B = 0$, then $\zeta_0 = id,\ \zeta_1 = \xi_1,\ \zeta_2 = \xi_2$. Now let $B \neq 0$. Consider the case $A \neq 0$. Then
		$$
		\begin{pmatrix}
		A& B\\
		C& D
		\end{pmatrix} = \begin{pmatrix}
		A& 0\\
		C& D - \frac{B}{A}C
		\end{pmatrix}
		\begin{pmatrix}
		1& \frac{B}{A}\\
		0& 1
		\end{pmatrix}.
		$$
		Let 
		$$\theta_1\colon
		\begin{pmatrix}
		u\\
		v
		\end{pmatrix}\mapsto
		\begin{pmatrix}
		1& \frac{B}{A}\\
		0& 1
		\end{pmatrix}
		\begin{pmatrix}
		u\\
		v
		\end{pmatrix},\qquad
		\theta_2\colon
		\begin{pmatrix}
		u\\
		v
		\end{pmatrix}\mapsto
		\begin{pmatrix}
		A& 0\\
		C& D - \frac{B}{A}C
		\end{pmatrix}
		\begin{pmatrix}
		u\\
		v
		\end{pmatrix}.
		$$
		If $\xi_1 = (u + f(v), v)$, then $\theta_1\circ\xi_1 = \left(u + f(v) + \frac{B}{A}v, v\right)$ is an elementary automorphism. We can put $\zeta_0 = id,\ \zeta_1 = \theta_1\circ\xi_1,\ \zeta_2 = \theta_2$.\\
		
		Now let $\xi_1 = (u, v + f(u))$. 
		Then, 
		$$\theta_1\circ\xi_1 = \left(u + \frac{B}{A}v + \frac{B}{A}f(u), v + f(u)\right) = \eta_2\circ\eta_1\circ\eta_0,$$ 
		where 
		$$ \eta_0 = \left(u + \frac{B}{A}v, u\right),\  {\eta_1} = \left(u+ \frac{B}{A}f(v),v\right),\ \eta_2 = \left(u, \frac{A}{B}(u - v)\right).$$
		Let's check this:
		$$
		\begin{pmatrix}
		u\\
		v
		\end{pmatrix}
		\xrightarrow{\eta_0}
		\begin{pmatrix}
		u + \frac{B}{A}v\\
		u
		\end{pmatrix}
		\xrightarrow{\eta_1}
		\begin{pmatrix}
		u + \frac{B}{A}v + \frac{B}{A}f(u)\\
		u
		\end{pmatrix}
		\xrightarrow{\eta_2}
		\begin{pmatrix}
		u + \frac{B}{A}v + \frac{B}{A}f(u)\\
		v + f(u)
		\end{pmatrix}.
		$$
		So, $\xi_2\circ\xi_1 = \theta_2\circ\theta_1\circ\xi_1 = \theta_2\circ\eta_2\circ\eta_1\circ\eta_0$.
		Let us put $\zeta_0 = \eta_0,\ \zeta_1 = \eta_1,\ \zeta_2 = \theta_2\circ\eta_2$.
		
		Now we concider the case $A = 0$. We have
		$$
		\begin{pmatrix}
		0& B\\
		C& D
		\end{pmatrix} =
		\begin{pmatrix}
		B& 0\\
		D& C
		\end{pmatrix}
		\begin{pmatrix}
		0& 1\\
		1& 0
		\end{pmatrix}.$$ 
		Denote 
		$$\theta_1 \colon 
		\begin{pmatrix}
		u\\
		v
		\end{pmatrix}\mapsto
		\begin{pmatrix}
		0& 1\\
		1& 0
		\end{pmatrix}\begin{pmatrix}
		u\\
		v
		\end{pmatrix},\qquad
		\theta_2\colon\begin{pmatrix}
		u\\
		v
		\end{pmatrix}\mapsto 
		\begin{pmatrix}
		B& 0\\
		D& C
		\end{pmatrix}
		\begin{pmatrix}
		u\\
		v
		\end{pmatrix}.
		$$
		So, $\xi_2\circ\xi_1 = \theta_2\circ\theta_1\circ\xi_1 = \theta_2\circ\xi_1'\circ\theta_1$, where ${\xi_1'} = (u, v + f(u))$, if $\xi_1 = (u + f(v), v)$ and~$\xi_1' = (u + f(v), v)$, if $\xi_1 = (u, v + f(u))$. In this case we put $\zeta_0 = \theta_1,\ \zeta_1 = \xi_1',\ \zeta_2 = \theta_2$.
	\end{proof}
	
	\begin{prop}\label{re}
		Suppose $\widehat{q} = 1$. Then any graded automorphism of $\A$ is graded-tame.
	\end{prop}
	\begin{proof}
		Let $\varphi = (f, g, z) \in E.$ Denote $\widetilde{\varphi} = \alpha(\varphi)$.
		Since $\widehat{q} = 1$, we have $b \equiv a (\mathrm{mod}\ c)$, and $b < a$. 
		This means that $L(f) = \lambda x, \lambda \neq 0$. Hence $L(\widetilde{\varphi}(u)) = \lambda u$. By Proposition~\ref{tgt},~$\widetilde{\varphi}$~ can be decomposed into a composition of linear and elementary automorphisms also preserving the grading: $\widetilde{\varphi} =\xi_n\circ\ldots\circ \xi_1$. We can assume that $\xi_1$ is linear. Let us prove by induction on $2\leq t\leq n-1$ that $\widetilde{\varphi}$ can be decomposed into another composition $\xi'_n\circ\ldots\circ \xi'_1$ of graded linear and elementary automorphisms such that for all $n-t\leq j\leq n$ we have $L(\xi'_j(u))=\lambda_j u$ and $\xi_1$ is linear.\\
		
		{\it Base.}\\
		For $t = 0$ the assertion is obvious.\\
		
		{\it Step.} Let us put $k = n - t$.
		By inductive hypothesis, we have $\widetilde{\varphi} =\xi_n\circ\ldots\circ \xi_1$ and for all $k < j\leq n$ we have $L(\xi_j(u))=\lambda_j u$ and $\xi_1$ is linear.
		
		If $\xi_{k}$ is elementary, then either $L(\xi_{k}(u))= u$ or $\xi_{k}=(u + \lambda v + f(v), v)$. In the second case let us put $\tau=(u + \lambda v, v)$. We can replase $\xi_{k}$ by $\xi_{k}'=\xi_{k}\circ \tau^{-1}$ and $\xi_{k-1}$ by $\xi_{k-1}'=\tau\circ\xi_{k-1}$. In both cases $L(\xi_{k-1}'(u))= u$.
		
		If $\xi_{k}$ is linear and $k>1$, then we can assume that $\xi_{k-1}$ is elementary and $k>2$. By Lemma \ref{tra}, the composition $\xi_{k}\circ\xi_{k-1}$ can always be rewritten as $\zeta_2\circ\zeta_1\circ\zeta_0$, where $\zeta_0$ and $\zeta_2$ are linear, $L(\zeta_2(u)) = \lambda' u$, and $\zeta_1$ is elementary. Then we replace $\xi_{k}\circ\xi_{k-1}\circ\xi_{k-2}$ by $\xi_{k}'\circ\xi_{k-1}'\circ\xi_{k-2}'$,\ where $\xi'_{k} = \zeta_2,\ \xi'_{k-1} = \zeta_1,\ \xi'_{k-2} = (\zeta_0\circ\xi_{k-2})$. Then $L(\xi_{k}'(u))=L(\zeta_2(u)) = \lambda'_{k} u$. Finally if $k = 1$ by induction hypothesis we assume that $\xi'_n\circ\ldots\circ\xi'_2(u) = \lambda_2\cdot\ldots\cdot\lambda_n u$. Since $L(\widetilde{\varphi}(u)) = \lambda u$ we have $\xi'_1(u) = \frac{\lambda}{\lambda_2\cdot\ldots\lambda_n}u$. 
		\\
		\\
		Now $L(\xi_k(u))=\lambda_k u$ so graded elementary automorphism $\xi_k$ can be lifted to graded elementary automorphism of the space by Lemma \ref{ma}
	\end{proof}
	
	Let $I = (u,v) \lhd \K[u,v]$ be the ideal generated by $u$ and $v$.
	
	\begin{lem}\label{ga}
		Let $\varphi \in \mathrm{E}$ is graded-tame automorphism and let $\widetilde{\varphi} = \alpha(\varphi)$. Then $\widetilde{\varphi}(u) = \lambda u + G$, where $G \in I^{\widehat{q} + c}$.
	\end{lem}
	\begin{proof}
		Let $\varphi = \varphi_n\circ\ldots\circ\varphi_1$, were $\varphi_k$ is elementary graded automorphism. Then $\widetilde{\varphi} = \xi_n\circ\ldots\circ\xi_1$, were $\xi_k = \alpha(\varphi_k)$. We can assume that $\xi_k = (\lambda_k u + \mu_k v^q, \nu_k v)$ or $(\lambda_k u, \mu_k v + \nu_k u^p)$, see Remerk~\ref{mr}. By Lemma \ref{ma} if a graded automorphism of the form $(\lambda_k u + \mu_k v^q, \nu_k v)$ can be lifted to an automorphism of $\mathcal{A}$, then $q \geqslant \widehat{q} + c$.	  Since $\xi_k(v)$ and $\xi_k(u)$ have zero free terms, we have $\xi_n\circ\ldots\circ\xi_1(v)\in I$, $\xi_n\circ\ldots\circ\xi_1(u)\in I^{\widehat{q}+c}$.
	\end{proof}
	
	\begin{prop}\label{wa}
		If $\widehat{q}\geq 2$, then the grading $\varGamma$ admits a graded-wild automorphism of the algebra $\K[x,y,z]$.
	\end{prop}
	\begin{proof}
		Let $\widehat{l}$ be the minimal $l$ such that the automorphism $(u, v + u^l)$ is $\widetilde{\varGamma}$-graded. If $\widehat{l} \geq c$, then the automorphism $(u, v + u^{\widehat{l} - c})$ also preserves the grading $\widetilde{\varGamma}$. Hence $\widehat{l} < c$. Consider the automorphisms~$\tau = (u + v^{\widehat{q}}, v)$ and $\phi = (u, v + u^{\widehat{l}})$. Let us show that  $\varepsilon=\tau^{-1}\circ\phi\circ\tau$ is a graded-wild automorphism. Let $I = (u,v) \lhd \K[u,v]$ be the ideal generated by $u$ and $v$. Then $$\varepsilon(u) = u + v^{\widehat{q}}-(v + (u + v^{\widehat{q}})^{\widehat{l}})^{\widehat{q}} = u - \widehat{q}v^{\widehat{q} - 1}u^{\widehat{l}} + F,$$
		where the polynomial $F \in I^{\widehat{l} + \widehat{q}}$.
		By Lemma \ref{ma} automorphism $\varepsilon$ can be lifted to an automorphism of the space and $\deg(v^{\widehat{q} - 1}u^{\widehat{l}}) = \widehat{q} - 1 + \widehat{l} < \widehat{q} + c$ so by Lemma \ref{ga} automorphism $\varepsilon$ is wild.
	\end{proof}
	\begin{remark}\label{co}
		It is easy to see that $\widehat{q} \geq 2$ if and only if there are integers $p \geq 1$ and $q\geq 2$ such as $a = qb + pc$.
	\end{remark}
	It remains to consider the cases when $ a, b $ or $ c $ can be equal to zero.
	\begin{lem}\label{zr}
		Suppose at least one number among $a$, $b$ and $c$ equals zero. Let $(a,b,c) \neq 0$. Then there are no graded-wild automorphisms.
	\end{lem}
	\begin{proof}
		If $c = 0, a > b$, then $z \mapsto \lambda z, y \mapsto \mu y$, and hence, by Lemma \ref{usl}, all automorphisms are tame. If $a = b, c = 0$, then any automorphism of $\varphi$ has the form $\varphi = (A(z)x + B(z)y,\ C(z)x + D(z)y,\ \kappa z + \mu)$, and the matrix $\Lambda = \begin{pmatrix}
		A(z)&B(z)\\
		C(z)&D(z)
		\end{pmatrix}$ is non-degenerate for all $z$, so $\det\varLambda = A(z)D(z) - C(z)D(z) = \lambda \in \K^{\times}$. Hence the greatest common divisor $A(z)$ and $C(z)$ in the ring $\K[z]$ is equal to one. We apply Euclid's algorithm to $A(z)$ and $C(z)$:\\
		\begin{align*}
		A(z) &= C(z)q_0 + r_1,\\
		B(z) &= r_1q_1 + r_2,\\
		&\ldots\\
		r_{k-2} &= r_{k-1}q_{k-1} + r_k,\\
		&\ldots\\
		r_{n-2} &= r_{n-1}q_{n-1} + r_n,\\
		r_{n-1} &= r_nq_n,
		\end{align*}
		where $q_k, r_k \in \K[z]$.
		Now, by the hand-generated automorphism $(x,y,\kappa^{-1}z - \mu\kappa^{-1})$ we reduce $\varphi$ to the form $\varphi = (A(z)x + B(z)y,\ C(z)x + D(z)y,\  z)$. Then by composition of elementary automorphisms of the form $(x - yq_k, y, z), (x , y - xq_k, z)$ we get that $\varphi$ has the form $(x + \widetilde{B}(z)y, \widetilde{C}y, z)$. It is clear that this automorphism decomposes into a composition of elementary automorphisms.
		Now, if $b = 0, a \neq 0, c \neq 0$, then $x \mapsto \lambda_1x, z \mapsto \lambda_2z$, hence, from Lemma \ref{usl}, all graded automorphisms are graded-tame.\\
		Now let $a \neq 0, b = c = 0$. Then $x \mapsto \lambda x$; the images of the variables $y$ and $z$ for any automorphism do not depend on $x$, and hence, by Theorem \ref{Jung}, all graded automorphisms are graded-tame.\\
		In the case of strictly positive and strictly negative gradings, all graded automorphisms are graded-tame by Proposition \ref{tetriod}.
	\end{proof}
	
	\begin{proof}[Proof of Theorem \ref{mainmain}]
		By Proposition \ref{wa} and Remark \ref{co} grading $\varGamma$ admits graded-wild automorphisms if $(\deg_{\varGamma}(x), \deg_{\varGamma}(y), \deg_{\varGamma}(z)) = (a,b,-c)$, $a,b,c$ strictly positive and $a = qb + pc$ for some integers $p \geq 1,\ q\geq 2$.\\
		By Remark \ref{pm},
		Proposition \ref{re},
		Proposition \ref{sz},
		Lemma \ref{zr} other gradings doesn't admit graded-wild automorphisms if $(a,b,c) \neq 0$.\\
		Finally case $(\deg_{\varGamma}(x), \deg_{\varGamma}(y), \deg_{\varGamma}(z)) = (0,0,0)$ admits wild automorphisms, see~\cite{b2}.
	\end{proof}

\end{document}